\documentclass[12pt]{amsart}
\usepackage{amscd,amssymb,epsfig,mathtools,accentset}
\calclayout

\numberwithin{equation}{section} 

\newcommand{\ud}{\,d} 
 
\newcommand{\R}{\mathbb{R}}

\renewcommand\P{{\mathcal P}}

\newcommand\cof{\operatorname{cof}}

\newcommand\diam{\operatorname{diam}}

\newcommand{\tir}[1]{\ensuremath{\overline {#1}}} 
\newtheorem{thm}{Theorem}[section] 
 
\newtheorem{lemma}[thm]{Lemma} 
 
\newtheorem{prop}[thm]{Proposition} 
\newtheorem{cor}[thm]{Corollary} 
\newtheorem{defn}[thm]{Definition} 
 
\newtheorem{rem}[thm]{Remark}

\def\whsq{\vbox to 5.8pt 
{\offinterlineskip\hrule 
\hbox to 5.8pt{\vrule height 
5.1pt\hss\vrule height 5.1pt}\hrule}}

\def\<{\langle} 
\def\>{\rangle} 
\def\PP{{\mathop{{\rm I}\kern-.2em{\rm P}}\nolimits}} 
\def\FF{{\mathop{{\rm I}\kern-.2em{\rm F}}\nolimits}}   
\def\ZZ{{\mathop{{\rm I}\kern-.2em{\rm Z}}\nolimits}} 
 
\newlength{\sidemargin} 
\begin{document}
\title[]{
Standard finite elements for the numerical resolution of the elliptic Monge-Amp\`ere equation: Aleksandrov solutions}


\author{Gerard Awanou}
\address{Department of Mathematics, Statistics, and Computer Science, M/C 249.
University of Illinois at Chicago, 
Chicago, IL 60607-7045, USA}
\email{awanou@uic.edu}  
\urladdr{http://www.math.uic.edu/\~{}awanou}

\maketitle

\begin{abstract}
We prove a convergence result for a natural discretization of the Dirichlet problem of the elliptic Monge-Amp\`ere equation using finite dimensional spaces of piecewise polynomial $C^0$ or $C^1$ functions. Standard 
discretizations of the type considered in this paper have been previous analyzed in the case the equation has a smooth solution and numerous numerical evidence of convergence were given in the case of non smooth solutions.
Our convergence result is valid for non smooth solutions, is given in the setting of Aleksandrov solutions, and consists in discretizing the equation in a subdomain with the boundary data used as an approximation of the solution in the remaining part of the domain. Our result gives a theoretical validation for the use of a non monotone finite element method for the Monge-Amp\`ere equation.


\end{abstract}

\section{Introduction}
In this paper we prove a convergence result for the numerical approximation of  solutions to  the Dirichlet problem for the Monge-Amp\`ere 
equation
\begin{equation}
\det  D^2 u = f  \ \text{in} \, \Omega, \, u=g \, \text{on} \, \partial \Omega, 
\label{m1}
\end{equation}
by elements of a space $V_h$ of  piecewise polynomial functions of some degree $k \geq 2$ which are either globally $C^0$ or globally $C^1$. 
The domain $\Omega \subset \mathbb{R}^d, d=2,3$ is assumed to be convex with polygonal 
boundary $\partial \Omega$. The expression $\det D^2 u$ should be understood in the sense of Aleksandrov c.f. section \ref{def-alex}. 
For a smooth function $u$, $D^2 u=\bigg( (\partial^2 u) / (\partial x_i \partial x_j)\bigg)_{i,j=1,\ldots, d} $ 
is the Hessian of $u$
and $f,g$ are given functions on $\Omega$ satisfying $f \in C(\Omega)$ with $0 < c_0 \leq f \leq c_1$ for constants  $c_0, c_1 \in \R$. 
We assume that $g \in C(\partial \Omega)$ can be extended to a function $\tilde{g} \in C(\tir{\Omega})$ which is  convex on $\Omega$. 

Let $f_m, g_m \in C^{\infty}(\tir{\Omega})$ such that $0 < c_2 \leq f_m \leq c_3$, $f_m$ converges uniformly to $f$ on $\tir{\Omega}$  and $g_m$ converges uniformly to $\tilde{g}$ on $\tir{\Omega}$. See for example \cite{MongeC1Alex}. Let $u_m \in C(\tir{\Omega})$ denote the 
Aleksandrov solution of the problem
\begin{equation}
\det  D^2 u_m = f_m  \ \text{in} \, \Omega, \, u_m=g_m \, \text{on} \, \partial \Omega. 
\label{m1m}
\end{equation}
Finally let $\widetilde{\Omega}$ be a convex polygonal subdomain of $\Omega$. We prove that the problem: 
find $u_{h} \in V_h(\widetilde{\Omega})$, $u_{h}=u_{m }$ on $\tir{\Omega} \setminus \widetilde{\Omega}$ and
\begin{equation} \label{main-problem}
\sum_{K \in \mathcal{T}_h(\widetilde{\Omega}) } \int_{K  } (\det  D^2 u_{h} - f_{ m} ) v_h \ud x=0, \forall v_h \in V_h(\widetilde{\Omega}) \cap H_0^1(\widetilde{\Omega}),
\end{equation}
has a (locally unique) piecewise strictly convex solution $u_h$ on $\widetilde{\Omega}$ which converges uniformly on compact subsets of $\widetilde{\Omega}$ to the solution $\tilde{u}$ of 
\begin{equation}
\det  D^2 \tilde{u} = f_{m}  \ \text{in} \, \widetilde{\Omega}, \, \tilde{u}=u_{m} \, \text{on} \, \tir{\Omega} \setminus \widetilde{\Omega}, 
\label{m2}
\end{equation}
which is convex on $\widetilde{\Omega}$ and continuous up to the boundary of $\widetilde{\Omega}$.

Here $\mathcal{T}_h(\widetilde{\Omega})$ denotes a quasi-uniform triangulation of the domain $\widetilde{\Omega}$ and  $V_h(\widetilde{\Omega})$ denotes a finite element space on $\widetilde{\Omega}$ of piecewise polynomial $C^r$ functions of degree $k$ for $r=0,1$. We make the abuse of notation of writing $u_{h}=u_{m}$ to mean that our approximations are discontinuous on the boundary and that $u_h$ coincides with $u_{m}$ at the Lagrange points on $\partial \widetilde{\Omega}$. For simplicity, we do not indicate the dependence of $\tilde{u}$ on $m$.

\subsection{Relevance of the convergence result for practical computations} \label{relevance}

Problems in affine geometry motivated the study of the Dirichlet problem for the Monge-Amp\`ere equation. See for example \cite{Baginski96} for a numerical study of the Gauss-curvature equation which is a Monge-Amp\`ere type equation. The Monge-Amp\`ere equation also appears in several applications, e.g. optimal transport and reflector design, but with the so-called second boundary condition, a term used to indicate that this type of boundary condition was studied much later than the Dirichlet problem. Formally, the numerical study of the second boundary condition can be reduced to a sequence of Dirichlet problems using a simple gradient descent. 

Recently, several researchers have used a standard discretization of the type considered in this paper for the numerical study of the reflector design problem \cite{Brix}. Even if one uses the same type of discretization for the Dirichlet problem \eqref{m1}, there is not yet a convergence theory. The convergence result of this paper addresses this issue.

Let $\delta >0$. It is known, c.f. Theorem \ref{subdomain-cvg}, that the  Aleksandrov solution $u_m$ of \eqref{m1m} converges uniformly on compact subsets of $\Omega$ to the  Aleksandrov solution $u$ of \eqref{m1}. We choose $m$ such that $| u(x) - u_{m}(x)| < \delta/2$ for all $x \in \widetilde{\Omega}$.
By unicity of the Aleksandrov solution $u_m$ of \eqref{m1m}, we have $\tilde{u}=u_{m}$ in $\widetilde{\Omega}$. Thus on each compact subset of $\widetilde{\Omega}$, $|u_h - u_m| < \delta/2$ for $h$ sufficiently small. The solution $u$ of \eqref{m1} can then be approximated within a prescribed accuracy by first choosing $m$ and then $h$ sufficiently small. We emphasize that the solution $\tilde{u}$ of \eqref{m2} is not necessarily smooth.

It remains to chose the data to compute the local solution of \eqref{main-problem}. We may assume that $| f(x) - f_{m}(x)| < \delta$, $| \tilde{g}(x) - g_{m}(x)| < \delta$ and
since $u_m=g_m$ on $\partial \Omega$ and $u_m \in C(\tir{\Omega})$, we may choose $\widetilde{\Omega}$ such that $|u_m - g_m| < \delta$ on $\tir{\Omega} \setminus \widetilde{\Omega}$. Thus, from a practical point of view, for the implementation, we see that one can take $\widetilde{\Omega}=\Omega$, $f_{m}=f$ with $u_h=g$ on $\partial \Omega$. 
A similar situation arises in the routine use in the finite element literature of the approximation of a smooth domain by a polygonal domain. 
Numerical experiments with the discretization considered in this paper were given in \cite{Awanou-Std01,AwanouPseudo10} for both smooth and non smooth solutions. For that reason, they are not reproduced here. Another possibility, but with results of less accuracy, is to actually implement the method on a subdomain. This can be easily tested on a code for \eqref{m1} by extending $g$ to a larger domain $\hat{\Omega}$ and using the restriction of $g$ on $\partial \hat{\Omega}$ as boundary value. 
For the extension of the framework of this paper to the second boundary condition, only the choice of $f_m$ and $\widetilde{\Omega}$ is needed. We wish to address this in a separate work.

\subsection{Methodology} \label{methods}

The purpose of this section is to explain the need for regularization of the data and the need of a subdomain for our convergence result.  The methodology of this paper may be applied to other settings where one has numerical evidence of convergence for discretizations of \eqref{m1}. The general methodology consists in
\begin{enumerate}
\item[1-] Prove the convergence and local uniqueness of the solution of the discrete problem \eqref{main-problem} when \eqref{m1} has a smooth solution. See \cite{Awanou-Std01}. Under the assumption that the discrete problem \eqref{main-problem} has a solution which is piecewise strictly convex, prove local uniqueness using the continuity of the eigenvalues of a matrix as a function of its entries. See section \ref{uni-disc}.
\item[2-] Verify that the numerical method is robust enough to handle the standard tests for non smooth solutions. In \cite{Awanou-Std01,AwanouPseudo10}, we prove the convergence of iterative methods which preserve weakly convexity and their effectiveness in capturing a piecewise convex solution of \eqref{main-problem} was illustrated with numerical experiments. 
\item[3-] Choose $m$, $f_m, g_m$ and $\widetilde{\Omega}$ as specified in section \ref{relevance}.
\item[4-] Consider  a sequence of smooth uniformly convex domains $\Omega_s$ increasing to $\Omega$ \cite{Blocki97}, with the property that $\widetilde{\Omega} \subset \Omega_s$ for all $s$, and the problems with smooth solutions \cite{Trudinger08}
\begin{align} \label{m1-ms}
\begin{split}
\det D^2 u_{m s}  = f_{m} \, \text{in} \, \Omega_s, 
u_{m s}  = g_{m}  \, \text{on} \, \partial \Omega_s.
\end{split}
\end{align}
From 
Theorem \ref{subdomain-cvg}, $u_{ms}$ converges uniformly on $\widetilde{\Omega}$ to $\tilde{u}$ as $s \to \infty$.
\item[5-] Establish that the discrete approximation $u_{ms,h}$ of the smooth function $u_{ms}$, on $\widetilde{\Omega}$ and with boundary data $u_{ms}$, converges uniformly to $u_{ms}$ on $\widetilde{\Omega}$ as $h \to 0$. This takes the form of an error estimate with constants depending on derivatives of $u_{ms}$.
\item[6-] 
Because $\widetilde{\Omega}$ is an interior domain, interior Schauder estimates allow to get a uniform bound on the derivatives of $u_{ms}$. In other words, $u_{ms,h}$ converges uniformly to $u_{m s}$ on compact subsets of $\widetilde{\Omega}$ at a rate which depends on $\widetilde{\Omega}$  but is independent of $s$. 

\item[7-] The local equicontinuity of piecewise convex functions allows to take a subsequence in $s$. This gives a piecewise convex finite element function $u_h$ which solves the finite element problem \eqref{main-problem}. The approximation $u_h$ is shown to converge uniformly on compact subsets of $\widetilde{\Omega}$ to the solution $\tilde{u}$ of \eqref{m2}. Local uniqueness of the discrete solution is a consequence of the work done in Step 1. 
\end{enumerate}

\subsection{Possible disadvantages of the approach in this paper}
We prove that \eqref{main-problem} has a piecewise strictly convex solution which is locally unique. Even when \eqref{m1} has a smooth solution,  global uniqueness of the discrete approximation has not been addressed in previous work. In the standard finite difference context, a variational approach presented in \cite{Aw-Matamba} allows to select a special discrete solution. Numerical results reported therein indicate that such an approach is effective when the right hand side of the Monge-Amp\`ere equation is a sum of Dirac masses. The analysis in \cite{Aw-Matamba} uses heavily results on the existence of local solutions. 

The convergence result in the paper 
uses results available for the approximation of smooth solutions of \eqref{m1} using standard discretizations. See for example \cite{Awanou-Std01}. When \eqref{m1} has a smooth strictly convex solution, these results say that the discrete problem has a solution for $h \leq h_0$ where $h_0 \to 0$ as a high order Sobolev norm of $u$ approaches infinity. Thus for example when $||u||_{C^{k+1}(\Omega)}$ is very big, existence of a discrete solution would hold for $h$ close to machine precision. And this is just for smooth solutions. The interior Schauder estimates give a possibly large upper bound on $||u_{ms}||_{C^{k+1}(\widetilde{\Omega})}$ as the latter depends on the distance of $\widetilde{\Omega}$ to $\Omega$. Thus it is not possible to quantify how small $h$ should be for the existence of $u_{ms,h}$. We recall that $u_{ms}$ and $u_{ms,h}$ were introduced in step 5 of the methodology described in section \ref{methods}.
The situation is similar to the one with monotone schemes, as results for the numerical approximation of viscosity solutions for \eqref{m1} in the degenerate case $f \geq 0$ are stated in terms of uniform convergence on compact subsets with no quantification of how small $h$ can be.

\subsection{Relation with other work}

A convergence analysis for a discretization of \eqref{m1} starts with a choice of a notion of weak solution. For an analysis based on the notion of viscosity solution, we refer to \cite{Oberman2010a}
 in the finite difference context, and to \cite{Feng2012} in the finite element context for radial solutions with a biharmonic regularization. 
 The discretization proposed in \cite{Oberman2010a} is a monotone scheme and thus enjoys a discrete maximum principle. One of the advantages of a monotone scheme is that one can prove existence of a discrete solution with no restriction on the size of the mesh. Nethertheless, the reader should be aware that there are many non monotone schemes for problems given in the setting of viscosity solutions e.g. \cite{Guermond08}.
The lack of a maximum principle for the discretizations analyzed in this paper is related to the difficulty of proving stability of the discretization for smooth solutions without assuming a bound on a high order norm of the solution. For that reason, we introduced the theoretical computational domain $\widetilde{\Omega}$ and fix the parameter $m$ in the regularization of the data.

The weak solution in the viscosity sense is known to be equivalent to the weak solution in the sense of Aleksandrov for $f \in C(\tir{\Omega})$ and $f >0$ on $\Omega$. The arguments of this paper are based on the notion of Aleksandrov solution. To the best of our knowledge, a proven convergence result for the numerical resolution of \eqref{m1} via the notion of Aleksandrov solution was only considered in \cite{Oliker1988} for the two dimensional problem. The approach in \cite{Oliker1988} uses geometric arguments and is different from the one taken here.

When the weak smooth solution of \eqref{m1} is a smooth strictly convex function, B$\ddot{\text{o}}$hmer \cite{Bohmer2008} studied $C^1$ approximations and his method has been implemented in \cite{Davydov12}. See also \cite{Davydov10}. B$\ddot{\text{o}}$hmer's method requires a modification of the Argyris space and numerical results in \cite{Davydov12} used Newton's method and did not address some of the standard test cases for non smooth solutions. In \cite{Brix}, it is shown that with a standard $C^1$ approximation based on $B$-splines, Newton's method coupled with trust region methods is effective for these standard test cases. Newton's method was also used in \cite{Feng2009} in the vanishing moment methodology. See also \cite{Awanou2009a}. In \cite{Awanou-Std01}, we analyzed the discretization \eqref{main-problem} for both $C^0$ and $C^1$ approximations and gave numerical evidence of convergence for non smooth solutions if one uses Lagrange elements and a time marching method. We previously gave the corresponding numerical results with $C^1$ approximations in \cite{AwanouPseudo10}.
In \cite{Lakkis11b} it is shown that Newton's method is effective if one uses a mixed formulation and implement the resulting method in primal form. See \cite{Lakkis11} for a description of the method for linear non variational problems. 
However in all these works, i.e. \cite{Bohmer2008,Davydov12,Brix,Feng2009,Awanou2009a,Lakkis11b,Awanou-Std01}, no proof of convergence is given in the case the solution of \eqref{m1} is not in $H^2(\Omega)$.

In this paper, we present a theory which explains why standard discretizations of the type considered in this paper do converge for non smooth solutions of the Monge-Amp\`ere equation. The easiest way to get insight into the problem, is through the approach which consists in regularizing the exact solution \cite{MongeC1Alex}. 
The latter approach is less general in the sense that it does not apply to collocation type discretizations such as the standard finite difference method.  In fact, it is a standard technique in the analysis of Aleksandrov solutions of the Monge-Amp\`ere equation, e.g. \cite[Lemma 3.1]{Figalli13b}, to regularize the data $f$, $g$ and take a sequence of smooth uniformly convex domains approximating the given domain. It is then natural, following principles of compatible discretization, that a similar approach can be followed for a discretization. The results we present are more natural with spaces of  piecewise polynomials  $C^1$ functions. These can be constructed using Argyris elements, the spline element method \cite{AwanouPseudo10} or isogeometric analysis. 
However standard Lagrange elements are more popular. In that case the results follow naturally from the ones with $C^1$ functions, as we show in this paper. 
Thus the main part of the paper is devoted to $C^1$ approximations.

Regularization of the data has been used in \cite{GlowinskiICIAM07}.  If one assumes that the domain $\Omega$ is smooth and uniformly convex, we can take $\widetilde{\Omega}=\Omega$ and use global Schauder estimates c.f. \cite{Trudinger08}, and a bootstrapping argument, to implement the compactness argument described in section \ref{methods}. To address the practical issue of dealing with curved boundaries, one should use the approach in \cite{Brenner2010b} which consists in a penalization of the boundary condition and the use of curvilinear coordinates for elements near the boundary. The boundary condition can now be taken as $\tilde{u}=g_m$. The approach of this paper can be easily adapted to explain the numerical results with singular data presented in \cite{Brenner2010b}. Yet for another example, again for a smooth uniformly convex domain, one can also use isogeometric analysis as in \cite{Awanou-Iso} for $\widetilde{\Omega}=\Omega$.

Without loss of generality, in subsequent papers on the analysis of schemes for \eqref{m1}, one may assume that $f$ and $g$ are smooth. In fact, one can even also assume that the solution is smooth, as the techniques of this paper can be applied to handle the non smooth case.

\subsection{Organization of the paper}

We organize the paper as follows. In the next section, we introduce some notation, recall the main results on the convergence of the discretization \eqref{main-problem} when \eqref{m1} has a smooth solution and the notion of Aleksandrov solution of \eqref{m1}. In section \ref{exhaustion} we give preliminary results on smooth and polygonal exhaustions of the domain. In section \ref{cvg} we give the proof of existence of a convex solution of \eqref{main-problem} for $C^1$ approximations. 
In section \ref{cvg-C0} we prove the existence of a piecewise convex solution of \eqref{main-problem} for $C^0$ approximations. 
The proof of the convergence of the discretization is given in section \ref{no-reg}.
Section \ref{sec-rem} groups some important remarks.  In particular the strict piecewise convexity of the discrete solution, and hence local uniqueness, is given in Remark \ref{uni-disc}.
The proof of some technical results are given in section \ref{appendix}.

\section{Notation and preliminaries}

\subsection{General notation}

For two subsets $S$ and $T$ of $\R^d$, we use the usual notation $d(S,T)$ for the distance between them. Moreover, $\diam S$ denotes the diameter of $S$. 

We use the standard notation for the Sobolev spaces $W^{k,p}(\Omega)$ with norms $||.||_{k,p,\Omega}$ and semi-norm $|.|_{k,p,\Omega}$. In particular,
$H^k(\Omega)=W^{k,2}(\Omega)$ and in this case, the norm and semi-norms will be denoted respectively by 
$||.||_{k,\Omega}$ and $|.|_{k,\Omega}$. 
When there is no confusion about the domain $\Omega$, we will omit the subscript $\Omega$ in the notation of the norms and semi-norms. We recall that $H^1_0(\Omega)$ is the subspace of $H^1(\Omega)$ of elements with vanishing trace on $\partial \Omega$.

We make the usual convention of denoting  constants by $C$ but will occasionally index some constants. 
We assume 
that the triangulation $\mathcal{T}_h(\Omega)$ of the domain $\Omega$ is shape regular in the sense that there is a constant $C>0$ such that for any element $K$, 
$h_K/\rho_K \leq C$, where $h_K$ denotes the diameter of $K$ and $\rho_K$ the radius of the largest ball contained in $K$. We also require the triangulation to be quasi-uniform in the sense that $h/h_{min}$ is bounded where $h$ and $h_{min}$ are the maximum and minimum respectively of $\{h_K, K \in \mathcal{T}_h \}$.

\subsection{Finite dimensional subspaces} \label{finite-dim}

We will need the broken Sobolev norms and semi-norms
\begin{align*}
|| v ||_{t,p,h} & = \bigg( \sum_{K \in \mathcal{T}_h(\Omega) } || v ||_{t,p,K}^2 \bigg)^{\frac{1}{2}}, 1 < p < \infty \\
|| v ||_{t,\infty,h} & = \max_{K \in \mathcal{T}_h(\Omega) } || v ||_{t,\infty,K},
\end{align*}
with a similar notation for $| v |_{t,p,h}$, $|| v ||_{t,h}$ and $| v |_{t,h}$.

We let $V_h(\Omega)$ denote a finite dimensional space of piecewise polynomial $C^r(\Omega)$ functions, $r=0,1$, of local degree $k \geq 2$, i.e., $V_h$ is a subspace of 
\begin{equation*} 
\{v \in C^r(\Omega), \ v|_K \in \P_k, \ \forall K \in \mathcal{T}_h(\Omega) \},
\end{equation*}
where $\P_k$ denotes the space of polynomials of degree less than or equal to $k$. 
We make the assumption that the following approximation properties hold:
\begin{equation}
|| v -\Pi_h v ||_{t,p,h} \leq C_{ap} h^{l+1-t} |v|_{l+1,p}, \label{schum}
\end{equation}
where $\Pi_h$ is a projection operator mapping the Sobolev space $W^{l+1,p}(\Omega)$  into $V_h$, $1 \leq p \leq \infty$ and $0 \leq t \leq l \leq k$. 
We require that the constant $C_{ap}$ does not depend on $h$ and $v$. 
We also make the assumption that the following inverse inequality holds
\begin{equation}
||v||_{t,p,h} \leq C_{inv} h^{l-t+\text{min}(0,\frac{d}{p}-\frac{d}{q})} ||v||_{l,q,h}, \forall v \in V_h, \label{inverse}
\end{equation}
and for $0 \leq l \leq t, 1 \leq p,q\leq \infty$. We require that the constant $C_{inv}$ be independent of $h$ and $v$. The approximation property and inverse estimate assumptions are realized for standard  finite element spaces  \cite{Brenner02}.


\subsection{Approximations of smooth solutions of the Monge-Amp\`ere equation} \label{sm-ma}

Next, we summarize the results of \cite{Awanou-Std01} of estimates for finite element approximations of smooth solutions of \eqref{m1}.

\begin{thm} \label{smooth-Omega-m}
Let $\Omega_s$ be a convex polygonal subdomain of $\Omega$ with a quasi-uniform triangulation $\mathcal{T}_h(\Omega_s)$. Assume that $u_s \in C^{\infty}(\tir{\Omega}_s)$ is a convex function which solves
\begin{equation*}
\det  D^2 u_s = f_s  \ \text{in} \, \Omega_s, \quad u_s=g_s \ \text{on} \, \partial \Omega_s, 
\end{equation*}
with $f_s, g_s  \in C^{\infty}(\tir{\Omega}_s)$ and $f_s \geq C >0$. 
We consider the problem: find $u_{s,h} \in V_h(\Omega_s)$, $u_{s,h}=g_{s}$ on $\partial \Omega_s$ and
\begin{equation} \label{m1hm}
\sum_{K \in \mathcal{T}_h(\Omega_s)}\int_{K} (\det  D^2 u_{s,h} - f_s ) v_h \ud x=0, \forall v_h \in V_h(\Omega_s) \cap H_0^1(\Omega_s).
\end{equation}
Problem  \eqref{m1hm} has a (locally unique) piecewise convex solution $u_{s,h}$ with
$$
||u_s-u_{s,h}||_{2,h,\Omega_s} \leq C_s h^{l-1}, 2 \leq l \leq k,
$$
and the constant $C_s$ is uniformly bounded if $||u_s||_{l+1,\infty,\Omega_s}$ is uniformly bounded.
\end{thm} 


For $C^1$ approximations, the result of Theorem \ref{smooth-Omega-m} follows from \cite[Theorems 5.1 and 8.7.]{Bohmer2008} and an inverse estimate. Equation \eqref{m1hm} differs from
\eqref{main-problem} in the sense that we assume here that $u_s$ is smooth whereas the solution $\tilde{u}$ of \eqref{m2} is not necessarily smooth.

\begin{cor} \label{unif-cvg}

Under the assumptions (and notation) of Theorem \ref{smooth-Omega-m}, the approximate solution $u_{s,h}$ converges uniformly on compact subsets of $\Omega_s$ to $u_s$ as $h \to 0$.

\end{cor}

\begin{proof}
For each element $K \in \mathcal{T}_h(\Omega_s)$, by the embedding of $H^2(K)$ into $L^{\infty}(K)$, we obtain
\begin{align*} 
||u_s-u_{s,h}||_{0,\infty,K} \leq ||u_s-u_{s,h}||_{2,K} \leq C_s h^{l-1}||u_s||_{l+1,\infty,\Omega_s}.
\end{align*}
Therefore
$$
||u_s-u_{s,h}||_{0,\infty,\Omega_s} \leq  C_s h^{l-1}||u_s||_{l+1,\infty,\Omega_s},
$$
and the result follows.
\end{proof}


\subsection{Interior Schauder estimates}

We will need estimates which depend on derivatives away from $\partial \Omega_s$ as we assume that $\Omega$ is a polygonal domain. This is the main reason for introducing the theoretical computational domain $\widetilde{\Omega}$. We will make the assumption that
$$
\tilde{\Omega} \subset \Omega_s, \ \text{for all} \ s,
$$
and thus the closure of $\widetilde{\Omega}$ is a compact subset of $\Omega$.

The proof of the following lemma is given in section \ref{appendix}.

\begin{lemma} \label{lem:int-Schauder}
We have the uniform interior Schauder estimates
$$
||u_{m s}||_{C^{k+1}(\widetilde{\Omega})} \leq C_{m},
$$
where $C_{m}$ depends only on $m, d, c_2, ||f_{m}||_{C^{k}(\tir{\Omega})},  \widetilde{\Omega}$ and $d(\widetilde{\Omega}, \partial \Omega)$.
\end{lemma}

\subsection{The Aleksandrov solution} \label{def-alex}

In this part of the section, we recall the notion  of Aleksandrov solution of \eqref{m1} and state several results that will be needed in our analysis. We follow the presentation in  \cite{Guti'errez2001} to which we refer for further details. 

Let $\Omega$ be an open subset of $\R^d$. Given a real valued convex function $v$ defined on $\Omega$, the normal mapping of $v$, or subdifferential of $v$, is a set-valued mapping $N_v$ from $\Omega$ to the set of subsets of $\R^d$ such that for any $x_0 \in \Omega$,
\begin{align*}
 N_v (x_0) = \{ \, q \in \R^d: v(x) \geq v(x_0) + q \cdot (x-x_0), \, \text{for all} \, x \in \Omega\,\}.
\end{align*}
Given $E \subset \Omega$, we define $N_v(E) = \cup_{x \in E} N_v(x)$ and denote by $|E|$ the Lebesgue measure of $E$ when the latter is measurable. 

If $v$ is a convex continuous function on $\Omega$, the class
 \begin{align*}
\mathcal{S} = \{\, E \subset \Omega, N_v(E) \, \text{is Lebesgue measurable}\, \},
\end{align*}
is a Borel $\sigma$-algebra and the set function $M[v]: \mathcal{S}  \to \tir{\R}$ defined by
$$
M[v] (E) = |N_v(E)|,
$$
is a measure, finite on compact sets, called the Monge-Amp\`ere measure associated with the function $v$. 

We are now in a position to define generalized solutions of the Monge-Amp\`ere equation. Let the domain $\Omega$ be open and convex. Given a Borel measure $\mu$ on $\Omega$, a convex function $v \in C(\Omega)$, is an Aleksandrov solution of
$$
\det D^2 v = \mu,
$$
if the associated Monge-Amp\`ere measure $M[v]$ is equal to $\mu$. If $\mu$ is absolutely continuous with respect to the Lebesgue measure and with density $f$, i.e. 
$$
\mu (B) = \int_B f \ud x, \, \text{for any Borel set} \, B,
$$
we identify $\mu$ with $f$. 
We have
\begin{thm}  [\cite{Hartenstine2006} Theorem 1.1] \label{ex-Alex}
Let $\Omega$ be a bounded convex domain of $\R^d$. Assume $f \in L^1(\Omega)$ and $g \in C(\partial \Omega)$ can be extended to a function $\tilde{g} \in C(\tir{\Omega})$ which is  convex in $\Omega$. Then the Monge-Amp\`ere equation \eqref{m1}
has a unique convex Aleksandrov solution in $C(\tir{\Omega})$.
\end{thm}


\begin{rem}
The assumption that  $g \in C(\partial \Omega)$ can be extended to a convex function $\tilde{g} \in C(\tir{\Omega})$ can be removed if the domain $\Omega$ is uniformly convex, \cite{Guti'errez2001}.
\end{rem} 

We recall that for a convex function $v$ in $C^2(\Omega)$, the Monge-Amp\`ere measure $M[v]$ associated with $v$ is given by
$$
M[v] (E) = \int_E \det D^2 v(x) \ud x,
$$
for all Borel sets $E \subset \Omega$. 

\begin{defn} A sequence $\mu_m$ of Borel measures is said to converge weakly to a Borel measure $\mu$ if and only if
$$
\int_{\Omega} p(x)\ud \mu_m  \to \int_{\Omega} p(x) \ud \mu,
$$
for every continuous function $p$ with compact support in $\Omega$.
\end{defn}
For the special case of absolutely continuous measures $\mu_m$ with density $a_m$ with respect to the Lebesgue measure, we have
\begin{defn}
Let $a_m, a \geq 0$ be given functions. We say that $a_m$ converges weakly to  $a$ as measures if and only if 
$$\int_{\Omega }a_m p \ud x \to \int_{\Omega } a p \ud x, $$ 
for all continuous functions $p$ with compact support in $\Omega$. 
\end{defn}
We have the following weak continuity result of Monge-Amp\`ere measures with respect to local uniform convergence.

\begin{lemma}[Lemma 1.2.3 \cite{Guti'errez2001}] \label{weak-s} Let $u_m$ be a sequence of convex functions in $\Omega$ such that $u_m \to u$ uniformly on compact subsets of $\Omega$. Then the associated Monge-Amp\`ere measures $M [u_m]$ tend to $M[u]$ weakly. 

\end{lemma}

\begin{rem} \label{s-c}
It follows that if $u_m$ is a sequence of  $C^2(\Omega)$ convex functions such that $u_m \to u$ uniformly on compact subsets of $\Omega$, with $u$ solving \eqref{m1}, then $\det D^2 u_m$ converges weakly to $f$ as measures. 
\end{rem}

We will often use the following lemma, the proof of which is given in section \ref{appendix}.

\begin{lemma} \label{Arzela}
Let $u_j$ denote a uniformly bounded sequence of convex functions on a convex domain $\Omega$. Then the sequence $u_j$ is locally uniformly equicontinuous and thus has a pointwise convergent subsequence.
\end{lemma}



\subsection{Approximations by solutions on subdomains}

For a function $b$ defined on $\partial \Omega$, we denote by $b^*$ its convex envelope, i.e. the supremum of all convex functions below $b$. If $b$ can be extended to a continuous convex function on $\tir{\Omega}$, then $b^*=b$ on $\partial \Omega$. 

Following \cite{Savin13}, we define a notion of convergence for functions defined on different subdomains. Recall that $\Omega \subset \R^d$ is bounded and convex. For a 
function $z: \Omega \to \R$, its upper graph $Z$ is given by
$$
Z : = \{ \, (x,x_{d+1}) \in \Omega \times \R, x_{d+1} \geq v(x) \, \}.
$$
For a function $b: \partial \Omega \to \R$, its upper graph is given by
$$
B: =  \{ \, (x,x_{d+1}) \in \partial \Omega \times \R, x_{d+1} \geq g(x) \, \}.
$$ 
\begin{defn}
We say that $z=b$ on $\partial \Omega$ if 
$$
B = \tir{Z} \cap (\partial \Omega \times \R).
$$
\end{defn}

\begin{defn} \label{haussdorff}
The Hausdorff distance between two nonempty subsets $K$ and $H$ of $\R^d$ is defined as
$$
\max \{ \, \sup [d(x,K), x \in H], \sup [d(x,H), x \in K]\, \}.
$$
\end{defn}

Let $\Omega_s \subset \Omega$ be a sequence of convex domains and let $z_s: \Omega_s \to \R$ be a sequence of convex functions on $\Omega_s$. We write $z_s \to z$ if the upper graphs $\tir{Z}_s$ converge in the Hausdorff distance to the upper graph $\tir{Z}$ of $z$. Similarly, for a sequence $b_s: \partial \Omega_s \to \R$, we say that $b_s \to b$ if the corresponding upper graphs converge in the Hausdorff distance. 

Finally, let $a_s: \Omega_s \to \R$ and $a: \Omega \to \R$. We write $a_s \to a$ if the $a_s$ are uniformly bounded and $a_s$ converges to $a$ uniformly on compact subsets of $\Omega$. 

To summarize, in Proposition \ref{savin} below, for a sequence of convex functions on $\Omega_s$ or for their restriction to $\partial \Omega_s$, the convergence is convergence of the corresponding upper graphs in the Hausdorff distance whereas for the data $a_s$ we use uniform convergence on compact subsets.

We have
\begin{prop} [Proposition 2.4 of \cite{Savin13}] \label{savin}
Let $z_s: \Omega_s \to \R$ be convex such that 
$$
\det D^2 z_s = a_s \, \text{in} \, \Omega_s, 
z_s = b_s \, \text{on} \, \partial \Omega_s.
$$
If
$$
z_s \to z, 
a_s \to a, 
b_s \to b,
$$
then
$$
\det D^2 z =a \, \text{in} \, \Omega, 
z = b^* \, \text{on} \, \partial \Omega,
$$
where $b^*$ denotes the convex envelope of $b$ on $\partial \Omega$. In particular if
 $b$ can be extended to a continuous convex function on $\tir{\Omega}$, $z=b$ on $\partial \Omega$.

\end{prop}


The following lemma will allow us to conclude that $b_s \to b$ when we know that $b_s$ converges uniformly to $b$. For the proof we refer to \cite[Exercise 9.40 ]{Giaquinta-Modica}.

\begin{lemma} \label{Giaquinta-Modica}
If $X$ is a compact metric space and $b_s: X \to \R$ converges uniformly to $b: X \to R$ on $X$, then the upper graph of $b_s$ converges to the upper graph of $b$ in the Hausdorff distance. 
\end{lemma}

We state an approximation result for Monge-Amp\`ere equations and give the proof in section \ref{appendix}.

\begin{thm} \label{subdomain-cvg}
Let $\Omega_s$ be a sequence of convex domains  increasing to $\Omega$, i.e. $\Omega_s \subset \Omega_{s+1} \subset  \Omega$ and $d(\partial \Omega_s, \partial \Omega) \to 0$ as $s \to \infty$. Assume that $z_s \in C(\tir{\Omega}_s)$ is a sequence of convex functions solving
$$
\det D^2 z_s = a_s \, \text{in} \, \Omega_s, z_s = b_s \, \text{on} \, \partial \Omega_s,
$$ 
with $a_s \geq 0$, $a_s, a \in C(\tir{\Omega})$. 
Assume that $a_s$ converges uniformly to $a$ on $\tir{\Omega}$, 
$b_s \in C(\tir{\Omega}_s)$, $b_s \to b$ uniformly on $\tir{\Omega}$ with $b \in C(\tir{\Omega})$ and convex on $\Omega$. 

Then $z_s$  converges (up to a subsequence) uniformly  on compact subsets of $\Omega$ to the unique convex solution $z$ of 
$$
\det D^2 z =a \, \text{in} \, \Omega, 
z = b \, \text{on} \, \partial \Omega,
$$
\end{thm}


\begin{rem} \label{cvg-um}
Under the assumptions of the above theorem, with $\Omega_s=\Omega$ for all $s$, we get that the sequence $z_s$ converges uniformly on compact subsets to the solution $z$. 
\end{rem}

\begin{rem} \label{extension}
If $v_s$ is a sequence of (piecewise) convex functions which converge on $\Omega$ to a (piecewise) convex function $v$ with upper graph $V$, we can extend $v$ canonically to the boundary by taking the function on $\partial \Omega$ with upper graph $\tir{V} \cap \partial \Omega \cap \R$. 
\end{rem}

\subsection{A characterization of weak convergence of measures}

The result we now give is well-known but we give a proof in section \ref{appendix} for completeness.

Let $C_b(\Omega)$ denote the space of bounded continuous functions on $\Omega$. We have
\begin{lemma} \label{weakH0}
Let 
$a_m, a \in C_b(\Omega), a_m, a \geq 0$ for $m=0,1,\ldots$ Assume that the sequence $a_m$ is uniformly bounded on $\Omega$ and that $a_m$ converges weakly to $a$ as measures and let $p \in H^{1}_0(\Omega)$. We have
$$
\int_{\Omega} a_m p \ud x \to \int_{\Omega} a p \ud x,
$$
as $m \to \infty$.
\end{lemma}

\subsection{Useful facts about convex functions}

It is known that the pointwise limit of a sequence of convex functions is convex. It follows that the pointwise limit of a sequence of piecewise convex functions is also piecewise convex.

Also, every pointwise convergent sequence of convex functions converges uniformly on compact subsets. See for example  \cite[Remark 1 p. 129 ]{Bakelman1994}. The result immediately extends to a sequence of piecewise convex functions.

\section{Smooth and polygonal exhaustions of the domain} \label{exhaustion}

It is known from \cite{Blocki97} for example that there exists a sequence of smooth uniformly convex domains $\Omega_s$ increasing to $\Omega$, i.e. $\Omega_s \subset \Omega_{s+1} \subset\Omega$ and $d(\partial \Omega_s, \partial \Omega) \to 0$ as $s \to \infty$. An explicit construction of the sequence $\Omega_s$ in the special case $\Omega=(0,1)^2$ can be found in \cite{Sulman11}.

Recall that $f_m$ and $g_m$ are $C^{\infty}(\tir{\Omega})$ functions such that  $0 < c_2 \leq f_m \leq c_3, f_m \to f$ and $g_m \to \tilde{g}$ uniformly on $\tir{\Omega}$. 
Thus the sequences $f_m$ and $g_m$ are uniformly bounded on $\Omega$. 
The sequences $f_m$ and $g_m$ may be constructed by extending the given functions to a slightly larger domain preserving the property $f \geq C >0$ for some constant $C$ and apply a standard mollification. See \cite{MongeC1Alex} for a different procedure. 
By  \cite{Caffarelli1984}, 
the problem \eqref{m1-ms}
has a unique convex solution $u_{m s} \in C^{\infty}(\tir{\Omega}_s)$. By Theorem \ref{subdomain-cvg}, as $s \to \infty$, the sequence $u_{m s}$  converges uniformly  on compact subsets of $\Omega$ to the unique convex solution $u_m \in C(\tir{\Omega})$ of Problem \eqref{m1m}. 
Moreover, the solution $u_m$ of \eqref{m1m} converges uniformly  on compact subsets of $\Omega$ to the unique convex solution $u$ of \eqref{m1}.

Recall that $\widetilde{\Omega}$ is a convex polygonal subdomain of $\Omega$ with a quasi-uniform triangulation $\mathcal{T}_h(\widetilde{\Omega})$.
We let $\delta >0$ be a fixed parameter and chose $m$ and $\widetilde{\Omega}$ such that 
 $| f(x) - f_{m}(x)| < \delta$, $| \tilde{g}(x) - g_{m}(x)| < \delta$  and $| u(x) - u_{m}(x)| < \delta$ for all $x \in \widetilde{\Omega}$. Without loss of generality we may assume that $\widetilde{\Omega} \subset \Omega_s$ for all $s$. 

We have


\begin{thm} \label{exh-arg} 
There exists a piecewise convex function $u_{h} \in V_h(\widetilde{\Omega})$ which is uniformly bounded on compact subsets of $\widetilde{\Omega}$ uniformly in $h$. The function $u_{h}$ 
satisfies $u_{h}=u_{m}$ on $\partial \widetilde{\Omega}$ and
is obtained as the limit of a subsequence in $s$ of the piecewise convex solution $u_{m s,h}$ in $V_h(\widetilde{\Omega})$ of the problem: 
\begin{equation} \label{m1h-sub-m}
\sum_{K \in \mathcal{T}_h}\int_{K\cap \widetilde{\Omega} } (\det  D^2 u_{m s,h} - f_{m} ) v_h \ud x=0, \forall v_h \in V_h(\widetilde{\Omega}) \cap H_0^1(\widetilde{\Omega}),
\end{equation}
with $u_{m s,h}= u_{m s}$ on $\partial \widetilde{\Omega}$. 
\end{thm}

\begin{proof}

Since $u_{m s}$ is smooth on $\tir{\Omega}_s$, Theorem \ref{smooth-Omega-m} yields a solution to Problem \eqref{m1h-sub-m}. Given a compact subset $K$ of $\widetilde{\Omega}$, we have 
\begin{align} \label{analogous}
||u_{m s} - u_{m s, h}||_{0,\infty,K} \leq ||u_{m s} - u_{m s, h}||_{0,\infty,\widetilde{\Omega}}  \leq C ||u_{m s}||_{k+1,\infty,\widetilde{\Omega}} \ h^{k-1}.
\end{align}
since $\widetilde{\Omega} \subset \Omega_s$. 
By the interior Schauder estimates Lemma \ref{lem:int-Schauder}, the sequence in $s$ of piecewise convex functions $u_{m s,h}$ is uniformly bounded on compact subsets, and hence by Lemma \ref{Arzela} has a convergent subsequence also denoted by $u_{ m s,h}$ which converges pointwise to a function $u_{h}$. 
The function $u_{h}$ is piecewise convex as the pointwise limit of piecewise convex functions and the convergence is uniform on compact subsets.

Next, we note that for a fixed $h$, $u_{m s,h}$ is a piecewise polynomial in the variable $x$ of fixed degree $k$ and convergence of polynomials is equivalent to convergence of their coefficients. Thus $u_{h}$ is a piecewise polynomial of degree $k$. Moreover, the continuity conditions on $u_{m s,h}$ are linear equations involving its coefficients. Thus $u_{h}$ has the same continuity property as $u_{m s,h}$. In other words $u_{h} \in V_h(\widetilde{\Omega})$.

Finally, since $u_{m s}$ converges uniformly on compact subsets to $u_{m}$ as $s \to \infty$, 
we have on $\partial \widetilde{\Omega}$, $u_{h}= u_{m}$ as $\partial \widetilde{\Omega}$ is by construction a compact subset of $\Omega$. 

As a consequence of the interior Schauder estimates, $u_{h}$ is uniformly bounded on compact subsets of $\widetilde{\Omega}$ uniformly in $h$.

\end{proof}

The goal of the next two sections is to prove that the function $u_{h}$ given by Theorem \ref{exh-arg} solves Problem \eqref{main-problem}. 




\section{Solvability of the discrete problems for $C^1$ approximations. } \label{cvg}

The goal of this section is to prove that \eqref{main-problem} has a solution in the case where the approximation space $V_h(\widetilde{\Omega})$ is a space of $C^1$ functions. Then Problem \eqref{m1h-sub-m} can be written
\begin{equation} \label{discrete-subd}
\int_{\widetilde{\Omega}} (\det  D^2 u_{m s,h} - f_{m} ) v_h \ud x=0, \forall v_h \in V_h(\widetilde{\Omega}) \cap H_0^1(\widetilde{\Omega}).
\end{equation}
To see that the left hand side of the above equation is well defined, one may proceed as  in \cite{AwanouPseudo10}. In addition the discrete solution $u_{m s,h}$ being piecewise convex and $C^1$ is convex on $\widetilde{\Omega}$, c.f. \cite[section 5 ]{Dahmen91}. We define
$$
f_{m s,h} = \det D^2 u_{m s,h}.
$$
We can then view $u_{m s,h}$ as the solution (in the sense of Aleksandrov) of the Monge-Amp\`ere equation
\begin{align*}
\det D^2 u_{m s,h} = f_{m s,h} \, \text{in} \, \widetilde{\Omega}. 
\end{align*}

By Lemma \ref{weak-s}, $\det D^2 u_{m s_l,h} \to \det D^2 u_{h}$ weakly as measures for a subsequence $s_l \to \infty$. Then by Lemma \ref{weakH0} we get for $v \in V_h(\widetilde{\Omega}) \cap H_0^1(\widetilde{\Omega})$,
\begin{equation} \label{det-cvg1}
\int_{\widetilde{\Omega}} (\det D^2 u_{m s_l,h}) v \ud x \to  \int_{\widetilde{\Omega}} (\det D^2 u_{h}) v \ud x. 
\end{equation}

It remains to prove that as $l \to \infty$
\begin{equation*} 
\int_{\widetilde{\Omega}} (\det D^2 u_{m s_l,h}) v \ud x \to  \int_{\widetilde{\Omega}} f_{m} v \ud x. 
\end{equation*}
This is essentially what is proved in the next theorem

 \begin{thm} \label{cvg-c1-app}
 Let $V_h(\widetilde{\Omega})$ denote a finite dimensional space of $C^1$ functions satisfying the assumptions of approximation property and inverse estimates of section \ref{finite-dim}. Then Problem \eqref{main-problem} has a convex solution $u_{h}$.  
 \end{thm}
 
\begin{proof}
Given $v \in V_h(\widetilde{\Omega}) \cap H_0^1(\widetilde{\Omega})$, let $v_l$ be a sequence of infinitely differentiable functions with compact support in $\widetilde{\Omega}$ such that $||v_l-v||_{1,2} \to 0$ as $l \to \infty$. We have by definition of
$f_{m s,h}$
\begin{align} \label{lim-unicity-p1}
 \int_{\widetilde{\Omega}} (\det D^2 u_{m s_l,h}) v \ud x  & =   \int_{\widetilde{\Omega}} f_{m s_l,h} v \ud x.
 \end{align}
 We have using \eqref{discrete-subd}
 \begin{align}  \label{lim-unicity-p2}
 \begin{split}
 \int_{\widetilde{\Omega}} f_{m s_l,h} v \ud x & = \int_{\widetilde{\Omega}} f_{m s_l,h} (v-v_l) \ud x + \int_{\widetilde{\Omega}} f_{m s_l,h} (v_l - \Pi_h (v_l) ) \ud x \\
 & \qquad \qquad + \int_{\widetilde{\Omega}} f_{m s_l,h} \Pi_h (v_l) \ud x \\
 & =  \int_{\widetilde{\Omega}} f_{m s_l,h} (v-v_l) \ud x + \int_{\widetilde{\Omega}} f_{m s_l,h} (v_l - \Pi_h (v_l) ) \ud x \\
 & \qquad \qquad + \int_{\widetilde{\Omega}} f_{m} \Pi_h (v_l) \ud x \\
 & =  \int_{\widetilde{\Omega}} (f_{m s_l,h} -f_{m}) (v-v_l) \ud x \\
  & \qquad \qquad + \int_{\widetilde{\Omega}} (f_{m s_l,h} -f_{m}) (v_l - \Pi_h (v_l) ) \ud x  + \int_{\widetilde{\Omega}} f_{m}  v \ud x.
 \end{split}
 \end{align} 
 By the inverse estimate \eqref{inverse} 
\begin{align*}
||\det D^2 u_{m s,h}||_{0,\infty,\widetilde{\Omega}} & \leq C || u_{m s,h}||_{2,\infty,\widetilde{\Omega}}^d \\
& \leq C h^{-2 d} || u_{m s,h}||_{0,\infty,\widetilde{\Omega}}^d.
\end{align*}
Hence by \eqref{int-Schauder2}
\begin{equation} \label{u-bound-det}
||\det D^2 u_{m s,h}||_{0,\infty,\widetilde{\Omega}} \leq C_{h},
\end{equation}
for a constant $C_{h}$ which depends on $h$ but is independent of $s$.

Since $f_{m}$ is uniformly bounded on $\tir{\Omega}$, it follows from \eqref{u-bound-det}
\begin{equation} \label{det-cvg3}
\bigg|  \int_{\widetilde{\Omega}} (f_{m s_l,h} -f_{m}) (v-v_l) \ud x \bigg| \leq C_{} ||v-v_l||_{1,2} \to 0 \, \text{as} \, l \to \infty.
\end{equation}
Finally, since $v \in V_h(\widetilde{\Omega})$, we have $\Pi_h (v) = v$ and hence
\begin{align*}
 \int_{\widetilde{\Omega}} (f_{m s_l,h} -f_{m}) (v_l - \Pi_h (v_l) ) \ud x &=  \int_{\widetilde{\Omega}} (f_{m s_l,h} -f_{m}) (v_l - v ) \ud x \\
 & \qquad  +  \int_{\widetilde{\Omega}} (f_{m s_l,h} -f_{m}) ( \Pi_h (v- v_l) ) \ud x.
\end{align*}
By Schwarz inequality, \eqref{u-bound-det} and \eqref{schum}
\begin{align*}
\bigg|  \int_{\widetilde{\Omega}} (f_{m s_l,h} -f_{m}) ( \Pi_h (v- v_l) ) \ud x \bigg| \leq C_{h} || \Pi_h (v- v_l)||_{0,2} \leq C_h  || v- v_l||_{1,2} \to 0 \, \text{as} \, l \to \infty. 
\end{align*}
Arguing again as in \eqref{det-cvg3}, it follows that
\begin{equation} \label{det-cvg4}
\int_{\widetilde{\Omega}} (f_{m s_l,h} -f_{m}) (v_l - \Pi_h (v_l) ) \ud x \to 0 \, \text{as} \, l \to \infty.
\end{equation}
 We conclude by \eqref{det-cvg1}--\eqref{det-cvg4} that as $l \to \infty$
\begin{equation*} 
\int_{\widetilde{\Omega}} (\det D^2 u_{m s_l,h}) v \ud x \to  \int_{\widetilde{\Omega}} f_{m} v \ud x. 
\end{equation*}
 By the unicity of the limit 
 $$
  \int_{\widetilde{\Omega}} (\det D^2 u_{h}) v \ud x =  \int_{\widetilde{\Omega}} f_{m} v \ud x. 
 $$
 That is, the limit $u_{h}$ solves \eqref{main-problem}. The existence of a solution to \eqref{main-problem} is proved. 
 \end{proof}

\section{Solvability of the discrete problems for $C^0$ approximations} \label{cvg-C0}
The arguments of the proof of Theorem \ref{cvg-c1-app} extends to the case of $C^0$ approximations to yield for $v \in V_h(\widetilde{\Omega}) \cap H_0^1(\widetilde{\Omega})$,
$$
\sum_{K \in \mathcal{T}_h }\int_{K \cap \widetilde{\Omega}} (\det D^2 u_{m s_l,h}) v \ud x \to \sum_{K \in \mathcal{T}_h }\int_{K \cap \widetilde{\Omega}}   f_{m} v \ud x. 
$$
It remains to show that as $l \to \infty$
$$
\sum_{K \in \mathcal{T}_h }\int_{K \cap \widetilde{\Omega}} (\det D^2 u_{m s_l,h}) v \ud x \to \sum_{K \in \mathcal{T}_h }\int_{K \cap \widetilde{\Omega}}   (\det D^2 u_{h}) v \ud x. 
$$
For this we need an extension of Lemma \ref{weakH0} to piecewise convex functions. This is the subject of Theorem \ref{w-p-c} below. We conclude that the analogue of Theorem \ref{cvg-c1-app} holds for $C^0$ approximations as well, i.e. the following theorem holds.

\begin{thm} \label{cvg-c0-app}
 Let $V_h(\widetilde{\Omega})$ denote a finite dimensional space of $C^0$ functions satisfying the assumptions of approximation property and inverse estimates of section \ref{finite-dim}. Then Problem \eqref{main-problem} has a piecewise convex solution $u_{h}$. 
 \end{thm}

Let $O$ be an open subset of $\R^d$. We recall that if $v \in C^2(O)$, the Monge-Amp\`ere measure $M[v]$ associated with $v$ on $O$ is given by $M[v](B) = \int_B \det D^2 v \ud x$ for any Borel subset of $O$. Also, if $v_m$ is  a sequence of convex functions on $O$ which converge uniformly to $v$ on compact subsets of $O$, $\det D^2 v_m \to \det D^2 v$ weakly as measures, i.e. 
$$
\int_K p(x) \det D^2 v_m(x) \ud x \to \int_K p(x) \det D^2 v(x) \ud x,
$$
for all continuous functions $p$ with compact support in $O$.

Equivalently, see \cite[section 1.9]{Evans-Gariepy} , 
\begin{align}
\begin{split} \label{weak-cvg-equiv}
M[v](A) & \leq \liminf_{m \to \infty} M[v_m](A), A \subset O,  A \, \text{open} \\
M[v](C) & \geq \limsup_{m \to \infty} M[v_m](C), C \subset O, C \, \text{compact}.
\end{split}
\end{align}

Assume $\Omega$ open and convex. We make the assumption that $\Omega$ is the finite union, indexed by $\mathcal{T}_h$, of closed subsets $K$ with nonempty interiors.  Let $v$ be a piecewise polynomial, piecewise $C^2$ on $\Omega$ and denote by $D^2 v$ (by an abuse of notation) its piecewise Hessian. 

We want to extend the weak convergence result of Monge-Amp\`ere measures  to piecewise convex functions. 
We first define new notions of Monge-Amp\`ere measures for piecewise convex functions.

\subsection{Partial normal mapping associated with a piecewise convex function}
Let $x_0 \in \Omega$ such that $x_0 \in \accentset{\circ}{K}$ for some $K \in \mathcal{T}_h$. We define
$$
N_v(x_0) = \{ \, q \in \R^d, q \in N_{v_{|\accentset{\circ}{K}}}(x_0) \, \text{for some}\, K \in \mathcal{T}_h\, \ \text{such that} \ x_0 \in \accentset{\circ}{K} \}.
$$
Thus $q \in N_v(x_0)$ if for all $K$ such that $x_0 \in \accentset{\circ}{K}$,
$$
  v(x) \geq v(x_0) + q \cdot (x-x_0), \, \text{for all} \, x \in \accentset{\circ}{K}.
$$
We do not define $N_v(x_0)$ for $x_0 \in \partial K$.  
Given $E \subset \Omega$, we define $$ 
N_v (E)= \sum_{K \in \mathcal{T}_h} N_v(E\cap \accentset{\circ}{K}), 
$$
and the partial Monge-Amp\`ere measure associated to a piecewise convex function $v$ as
\begin{align} \label{MA-p-c}
	M[v](E) = |N_v (E)| = \sum_{K \in \mathcal{T}_h } M[v|_{\accentset{\circ}{K}}]  (E\cap \accentset{\circ}{K}).
\end{align}
 If $v \in C^2(\accentset{\circ}{K} )$ and is convex on $\accentset{\circ}{K}$ for all $ K$, then
$$
M[v] (E) =   \sum_{K \in \mathcal{T}_h} \int_{E \cap \accentset{\circ}{K}} \det D^2 v(x) \ud x.
$$
We will also use the notation $\det D^2 v$ for $M[v] (E) $ when $v$ is piecewise convex.

\subsection{Weak convergence of partial Monge-Amp\`ere measures associated with piecewise convex functions}

Let $D \subset \Omega$ is compact. We claim that $D \cap \accentset{\circ}{K}$ is also compact in $\accentset{\circ}{K}$. Assume that $D \cap \accentset{\circ}{K} \subset \cup_{i \in I} U_i, U_i$ open in $\accentset{\circ}{K}$ for all $i$. We have $U_i$ open in $\Omega$ as well. Since $D$ and $K$ are closed, $D \cap K$ is closed and hence $\Omega \setminus D \cap K$ is open. Therefore $(\Omega \setminus D \cap K) \cup (\cup_{i \in I} U_i)$ is an open covering of $D$ which has a finite subcovering $(\Omega \setminus D \cap K) \cup (\cup_{j \in J} U_J)$. It follows that $\cup_{j \in J} U_J$ is a finite  subcovering of $D \cap \accentset{\circ}{K}$.

We also recall that for two sequences $a_m$ and $b_m$,
\begin{align}
\liminf (a_m + b_m) & \geq \liminf a_m + \liminf b_m \label{limf}\\
\limsup (a_m + b_m) & \leq \limsup a_m + \limsup b_m. \label{lims}
\end{align}

We claim
\begin{thm} \label{w-p-c}
Assume that $v_m$ is a sequence of piecewise convex, piecewise $C^2$ functions on $\Omega$ which converge uniformly on compact subsets of $\Omega$ to $v$, which is also then piecewise convex. Then
$$
\int_{\Omega} p(x) \det D^2 v_m(x) \ud x \to \int_{\Omega} p(x) \det D^2 v (x) \ud x,
$$
for all continuous functions $p$ with compact support in $\Omega$.
\end{thm}

\begin{proof} 
It is enough to prove \eqref{weak-cvg-equiv} with $O$ replaced by $\Omega$.

Let $A \subset \Omega$ be open. Then $A \cap  \accentset{\circ}{K}$ is open in both $\Omega$ and $ \accentset{\circ}{K}$. Since $v_m|_{\accentset{\circ}{K}}$ converges to $v|_{\accentset{\circ}{K}}$ uniformly on compact subsets of $\accentset{\circ}{K}$, we have
$$
M[v|_{\accentset{\circ}{K}}](A  \cap  \accentset{\circ}{K} )  \leq \liminf_{m \to \infty} M[v_m|_{\accentset{\circ}{K}}](A  \cap  \accentset{\circ}{K}).
$$
Thus by \eqref{MA-p-c} and \eqref{lims}
\begin{align*}
M[v] (A) & =   \sum_{K \in \mathcal{T}_h} M[v|_{\accentset{\circ}{K}}]  (A \cap \accentset{\circ}{K}) \leq \sum_{K \in \mathcal{T}_h}   \liminf_{m \to \infty} M[v_m|_{\accentset{\circ}{K}}](A  \cap  \accentset{\circ}{K}) \\
& \leq   \liminf_{m \to \infty}  \sum_{K \in \mathcal{T}_h} M[v_m|_{\accentset{\circ}{K}}](A  \cap  \accentset{\circ}{K})  =   \liminf_{m \to \infty} M[v_m](A). 
\end{align*}

Next, let $C \subset \Omega$ be compact. We recall that  $C \cap  \accentset{\circ}{K}$ is compact in $ \accentset{\circ}{K}$. Thus
$$
M[v|_{\accentset{\circ}{K}}](C  \cap  \accentset{\circ}{K} )  \geq \limsup_{m \to \infty} M[v_m|_{\accentset{\circ}{K}}](C  \cap  \accentset{\circ}{K}).
$$
Thus by \eqref{MA-p-c} and \eqref{limf}
\begin{align*}
M[v] (C) & =   \sum_{K \in \mathcal{T}_h} M[v|_{\accentset{\circ}{K}}]  (C \cap \accentset{\circ}{K}) \geq \sum_{K \in \mathcal{T}_h}   \limsup_{m \to \infty} M[v_m|_{\accentset{\circ}{K}}](C  \cap  \accentset{\circ}{K}) \\
& \geq   \limsup_{m \to \infty}  \sum_{K \in \mathcal{T}_h} M[v_m|_{\accentset{\circ}{K}}](C  \cap  \accentset{\circ}{K})  =   \limsup_{m \to \infty} M[v_m](C). 
\end{align*}
This completes the proof.
\end{proof}

\section{Convergence of the discretization} \label{no-reg}

We have 
\begin{thm} \label{noreg}
Under the assumptions set forth in the introduction, the piecewise convex solution $u_h$ of Problem \ref{main-problem} (given by Theorem \ref{cvg-c0-app}) converges uniformly on compact subsets of $\widetilde{\Omega}$, as $h \to 0$, to the solution $\tilde{u}$ of \eqref{m2} which is convex on $\widetilde{\Omega}$ and continuous up to the boundary. 
\end{thm}

\begin{proof}

We recall from Theorem \ref{exh-arg} that the function $u_h$ is uniformly bounded on compact subsets of $\widetilde{\Omega}$. It follows from Lemma \ref{Arzela} that there exists a subsequence $u_{h_l}$ which converges pointwise to a piecewise convex function $v$. The latter is continuous on $\widetilde{\Omega}$ as it is locally finite. Moreover the convergence is uniform on compact subsets of $\widetilde{\Omega}$.

Recall also from Theorem \ref{exh-arg} that $u_h$ is obtained as a subsequence in $s$ of the approximations $u_{m s,h}$ of smooth solutions $u_{m s}$ which converge to $u_{m}$ uniformly on compact subsets of $\Omega$.

Let $K$ be a compact subset of $\widetilde{\Omega}$. There exists a subsequence $u_{m, s_l, h}$ which converges uniformly to $u_h$ on $K$.  By the uniform convergence of $u_{m s}$ to $u_{m}$ 
on $K$, we may assume that $u_{m, s_l}$ converges uniformly to $u_{m}$ on $K$.

Let now $\epsilon >0$. 
Since $u_{h_l}$ converges uniformly on $K$ to $v$, $\exists l_0$ such that $\forall l \geq l_0$ $|u_{h_l}(x) - v(x)|< \epsilon/6$ for all $x \in K$.
 
There exists $l_1 \geq 0$ such that for all $l \geq \max \{ \,l_0,l_1 \, \}$, 
$|u_{m s_l,h_l}(x) - u_{h_l}(x)|< \epsilon/6$ for all $x \in K$.

Moreover, there exists $l_2 \geq 0$ such that for all $l \geq \max \{ \,l_0,l_1, l_2 \, \}$, $|u_{m s_l}(x) - u_{m}(x)|< \epsilon/6$ for all $x \in K$.

Similarly to \eqref{analogous}, we have on $K$,  $|u_{m s,h_l}(x) - u_{m s}(x)|\leq C_{m} h_l$ for all $x \in K$. We recall that the constant $C_{m}$ is independent of $s$ but depends also on $\widetilde{\Omega}$.

We conclude that for $ l \geq \max \{ \,l_0,l_1, l_2 \, \}$, $|u_{m}(x) - v(x) | <  \epsilon/2 + C h_l$ for all $x \in K$.
We therefore have for all $\epsilon >0 $ $|u_{m}(x) - v(x) | < \epsilon$. We conclude that $u_{m}=v$ on $K$. 

Since $u_h=u_{m}$ on $\partial \widetilde{\Omega}$, it follows that $v=u_{m}$ on $\partial \widetilde{\Omega}$. This proves that $u_{m}=v$ on $\widetilde{\Omega}$.

The limit $u_{m}$ being unique, we conclude that $u_h$ converges uniformly on compact subsets of $\widetilde{\Omega}$ to $\tilde{u}$.

\end{proof}

 \section{Remarks} \label{sec-rem}
  We make the abuse of notation of denoting by $D^2 w_h$ the piecewise Hessian of $w_h \in V_h(\widetilde{\Omega})$. 
 Let $\lambda_1(D^2 w_h)$ denotes the smallest eigenvalue of $D^2 w_h$.
 
 \subsection{Strict piecewise convexity of the discrete solution} \label{big-ass}
 
 
 Let $x_0 \in \widetilde{\Omega}$. If necessary, by identifying $u_h$ with $u_h + \epsilon |x-x_0|^2$, where $\epsilon$ is taken to be close to machine precision, we may assume that the solution $u_h$ is piecewise strictly convex. 
 
In the case of $C^0$ approximations, a direct argument can be given if we assume that  $k \geq d+1$. For $C^1$ approximations, we would need $k \geq 2(d+1)$.

Assume that $\det D^2 u_h$ (computed piecewise) is non zero on a set of non zero Lebesgue measure. Then since $\det D^2 u_h$ is a piecewise polynomial, it must vanish identically on an element $K_0$. 
Let $v$ denote the unique polynomial of degree $d+1$ which vanishes identically on all faces of $K_0$ and with average 1 on $K_0$. We denote as well by $v$ its extension by 0 on all other elements. Then $v >0$ in $K_0$ and $v \in V_h(\widetilde{\Omega}) \cap H_0^1(\widetilde{\Omega})$ and thus
$$
\int_{\widetilde{\Omega}} f v \ud x =\int_{K_0} f v \ud x >0.
$$
On the other hand
$$
\int_{\widetilde{\Omega}} f v \ud x  = \sum_{K \in \mathcal{T}_h}  \int_{K \cap \widetilde{\Omega}} (\det D^2 u_h) v \ud x = \int_{K_0} (\det D^2 u_h) v \ud x=0,
$$
since  $\det D^2 u_h=0$ on $K_0$. Contradiction. We therefore have $\det D^2 u_h > 0$ in $\widetilde{\Omega}$.

For $C^1$ approximations, we take $v^2$ as the test function and the same argument applies.
 
  \subsection{Uniqueness of the discrete solution} \label{uni-disc}
  In the case of $C^1$ approximations, there is a unique solution in a sufficiently small neighborhood of $u_{h}$ if, as in section \ref{big-ass}, one makes the assumption  that the solution is piecewise strictly convex. 
  
  Define $B_{\rho}(u_{h}) = \{ \, w_h \in V_h, ||w_h-u_{h}||_{2,\infty} \leq \rho\, \}$. Then since $\lambda_1(D^2 u_{h}) \geq c_{00}$, by the continuity of the eigenvalues of a matrix as a function of its entries, $w_h $ is strictly convex for $\rho$ sufficiently small and $\rho$ independent of $h$. 
  
  Let then $u_{h}$ and $v_{h}$ be two solutions of \eqref{main-problem} in $B_{\rho}(u_{h})$. By the mean value theorem, see for example \cite{AwanouPseudo10}, we have for $w_h \in V_h(\widetilde{\Omega}) \cap H_0^1(\widetilde{\Omega})$
  \begin{align*}
  0 & = \int_{\widetilde{\Omega}} (\det D^2 u_{h} - \det D^2 v_{h}) w_h \ud x \\
  & = 
  -\int_0^1 \bigg\{ \int_{\widetilde{\Omega}}
[ (\cof(1-t) D^2 v_{h} + t D^2 u_{h} (D u_{h} - D v_{h}) ] \cdot D w_h \ud x \bigg\} \ud t.
  \end{align*}
For each $t \in [0,1]$, $(1-t) v_{h} + t  u_{h} \in B_{\rho}(u_{h})$ and is therefore strictly convex, that is
$$
 (\cof(1-t) D^2 v_{h} + t D^2 u_{h}) D (v_{h} -  u_{h}) ] \cdot D (v_{h}-u_{h}) \geq C |v_{h} -u_{h}|_1^2, C >0.
$$
Since $u_{h}=v_{h}=u_m$ on $\partial \widetilde{\Omega}$, we have $v_{h}-u_{h}=0$ on $\partial \widetilde{\Omega}$ and so integrating both sides, we obtain $ |v_{h} -u_{h}|_1=0$. But $u_{h}=v_{h}=g$ on $\partial \widetilde{\Omega}$ and therefore $u_{h}=v_{h}$.

For the uniqueness of the $C^0$ approximation, one would have to repeat the fixed point argument of \cite{Awanou-Std01} which were written under the assumption that $u$ is smooth strictly convex. Similar arguments apply for 
$u_{h}$ and in $B_{\rho}(u_{h})$. One may refer to \cite{AwanouPseudo10} where similar arguments are implemented in the context of $C^1$ approximations.



  \subsection{The degenerate case $f \geq 0$ and the case of unbounded $f$}
We now argue that our assumption $0 < c_0 \leq f \leq c_1$ is not restrictive. 

For $M >0$, if one defines $f_M$ by
$$
f_M(x) = f(x) \ \text{for} \ f(x) \leq M, \ \text{and} \ 0 \ \text{otherwise},
$$
we showed in \cite{MongeC1Alex} how the Aleksandrov solution of \eqref{m1} is a limit of solutions of Monge-Amp\`ere equations with right hand side $f_M$ and boundary data $g$.

On the other hand, the constant $c_0$ may be assumed to be close to machine precision. Moreover, in the case $f$ bounded with $f \geq 0$, for $\epsilon >0$, it is a simple consequence of \cite[Lemma 5.1]{Hartenstine2006} that solutions of
Monge-Amp\`ere equations with right hand side $f+\epsilon$ and boundary data $g$ converge uniformly on compact subsets to the Aleksandrov solution of \eqref{m1} as $\epsilon \to 0$.  

\section{Appendix} \label{appendix}

We give in this section the proof of some technical results.

\begin{proof}[Proof of Lemma~\ref{lem:int-Schauder}]

In the homogeneous case, i.e. for $g_{m}=0$, the result can be inferred from \cite{Caffarelli1990}. See also
\cite[Theorem 2.16]{Figalli14}.

In the non homogeneous case, it seems that the only genuine interior Schauder estimates for \eqref{m1-ms}, with constant depending only on the diameter of the compact subset $ K \subset \Omega_s$ and not on $\Omega_s$
is to rely on the corresponding result for the complex Monge-Amp\`ere equation in \cite[Theorem 4]{Dinew}. See also the corresponding A.M.S. Mathematical Review. For the convenience of the reader, we finish the proof with a brief introduction to the complex Monge-Amp\`ere equation.

It follows from \cite[Theorem 4]{Dinew} that
$$
||u_{m s}||_{C^2(\widetilde{\Omega})} \leq C_{m},
$$
where $C_{m}$ depends only on $m, d, c_2, ||f_{m}||_{C^1(\tir{\Omega})}$ and $ d(\widetilde{\Omega}, \partial \Omega)$. The estimate for higher order derivatives follows from standard elliptic regularity arguments.
For example differentiating the equation one time, and taking into account the smoothness of $f_{m}$ and the $C^{2}$ estimate, one obtains a second order linear equation which, because of the strict convexity of the solution $u_{m s}$, is uniformly elliptic on compact subsets of $\Omega$ and with solution a first derivative of $u$. 
The interior Schauder estimates for uniformly elliptic linear equations \cite[Theorem 6.2]{Gilbarg2001} then applies to give the desired estimate for the third derivatives. Repeating this process is known as a bootstrapping argument.

Let us illustrate the technique with the two dimensional Monge-Amp\`ere equation
$$
u_{xx} u_{yy} - u_{xy}^2 = f(x,y),
$$
where we use another standard notation for derivatives for simplicity. 
Put $v=u_x$. Differentiating with respect to $x$, we get the second order equation
$$
u_{yy} v_{xx} + u_{xx} v_{yy} -2 u_{xy} v_{xy} = f_x. 
$$
By the strict convexity of $u$, the equation is uniformly elliptic and hence
$$
||v||_{C^2(\widetilde{\Omega})} \leq C,
$$
with $C$ depending on $\max_{\Omega} v$, $ ||f_{x}||_{C^1(\tir{\Omega})}, \widetilde{\Omega}$, $d(\widetilde{\Omega}, \partial \Omega)$, the smallest eigenvalue of $D^2 u$ and a bound on the $C^{2,\alpha}$ norm of $u$. The latter bound implies an upper bound on
the eigenvalues of $D^2 u$, and since $\det D^2 u =f \leq C$, we obtain a positive lower bound for the smallest eigenvalue of $D^2 u$. A similar argument applies to $u_y$ and thus $||u||_{C^3(\widetilde{\Omega})} \leq C$,
with $C$ depending only on $ ||f_{}||_{C^2(\tir{\Omega})}, \widetilde{\Omega}$ and $d(\widetilde{\Omega}, \partial \Omega)$.

We finish with a brief introduction to the complex Monge-Amp\`ere equation. First the domain $\Omega \subset \R^d$ is identified with a convex domain of $\mathbb{C}^d$. Let now $u$ be a strictly convex smooth solution and
put $z_i = x_i + \sqrt{-1} \, y_i, i=1, \ldots,d$. We can then view $u$ as a function of $z$ defined by $u(z) = u(x)$. Same for $f$ and $g$. The complex Monge-Amp\`ere equation is given by
\begin{align}\label{m1-C}
\begin{split}
\det \bigg( \frac{\partial^2 u}{\partial z_i \partial \tir{z}_j} \bigg)_{i,j=1,\ldots,d} & = f \ \text{in} \ \Omega \\
u & = g \ \text{on} \  \partial \Omega,
\end{split}
\end{align}
where
\begin{align*}
\frac{\partial u}{\partial z_i } & = \frac{1}{2} \bigg( \frac{\partial u}{\partial x_i } - \sqrt{-1} \, \frac{\partial u}{\partial y_i } \bigg)  \\
\frac{\partial u}{\partial \tir{z}_j } & = \frac{1}{2} \bigg( \frac{\partial u}{\partial x_j } + \sqrt{-1} \, \frac{\partial u}{\partial y_i } \bigg).
\end{align*}
This clearly reduces to \eqref{m1} for real-valued functions defined on a convex domain of $\R^n$. The analogue of convex solution is a plurisubharmonic function, i.e. a function for which the Hessian matrix in \eqref{m1-C} is positive.

\end{proof}

\begin{proof}[Proof of Lemma~\ref{Arzela}]
For $p_j \in \partial u_j(x)$ and $x \in \Omega$, we have by \cite[Lemma 3.2.1]{Guti'errez2001}
$$
|p_j| \leq \frac{|u_j(x)| }{d(x,\partial \Omega)} \leq \frac{C}{d(x,\partial \Omega)},
$$
for a constant $C$ independent of $j$. Arguing as in the proof of \cite[Lemma 1.1.6]{Guti'errez2001}, it follows that the sequence $u_j$ is uniformly Lipschitz and hence equicontinuous on compact subsets of $\Omega$. By the Arzela-Ascoli theorem, \cite[p. 179]{Royden}, we conclude that the result holds.

\end{proof}

\begin{proof}[Proof of Lemma~\ref{subdomain-cvg}]
By convexity of $z_s$, see \cite[Theorem 3.4.7]{Niculescu}, we have
$$
z_s(x) \leq \max_{ x \in \Omega_s}  b_s \leq \max_{ x \in \Omega}  b_s \leq C, \forall x \in \Omega_s,
$$
for a constant $C>0$.

Let now $C$ denote the minimum of $b_s$ on $\partial \Omega_s$. Since $z_s = b_s$ on $\partial \Omega_s$, we have $z_s-C \geq 0$ on $\partial \Omega_s$. Either $z_s(x) - C \geq 0$ for $x \in \Omega_s$, or
by 
Aleksandrov's maximum principle,  \cite[Lemma 3.5 ]{Rauch77} or \cite[Proposition 6.15]{GutierrezNguyen07}, 
$$
( -(  z_s(x) -C) )^n \leq c_n (\diam \Omega_s)^{n-1} d(x,\partial \Omega_s) \int_{\Omega_s} a_s \ud x,
$$ 
where $c_n$ is a constant which depends only on $n$. We note that $a_s$ is uniformly bounded on $\Omega$. It follows that the sequence $z_s$ is bounded below on $\Omega_s$.

By Lemma \ref{Arzela}, the sequence $z_s$ being bounded has a pointwise convergent subsequence, also denoted by $z_s$, to a limit function $z$. But since $z_s$ is a sequence of convex functions on $\Omega_s$, and $\Omega_s$ increases to $\Omega$, the limit function $z$ is a convex function on $\Omega$ and the convergence is uniform on compact subsets of $\Omega$. 
By \cite[Proposition 2.6]{Savin13} the corresponding upper graphs converge in the Hausdorff distance, i.e.  $z_s \to z$.

By Lemma \ref{Giaquinta-Modica} $a_s \to a$ and $b_s \to b$. Thus by Proposition \ref{savin}, or  \cite[Proposition 2.4 ]{Savin13}, we have
$$
\det D^2 z = a \, \text{in} \, \Omega, 
z = b \, \text{on} \, \partial \Omega.
$$ 
The proof is completed.

\end{proof}

\begin{proof}[Proof of Lemma~\ref{weakH0}]
Since  $p \in H^{1}_0(\Omega)$, there exists a sequence $p_l$ of infinitely differentiable functions with compact support in $\Omega$ such that $||p_l-p||_{1,2} \to 0$ as $l \to \infty$. 
We have
\begin{align*}
 \int_{\Omega} (a_m-a) p \ud x =   \int_{\Omega} (a_m-a) (p-p_l) \ud x +  \int_{\Omega} (a_m-a) p_l \ud x.
\end{align*}
By assumption $ \int_{\Omega} (a_m-a) p_l \ud x \to 0$ as $m \to \infty$. Moreover, since $\Omega$ is bounded and $||a_m||_{0,\infty} \leq C$ for all $m$, we have
\begin{align*}
\bigg|  \int_{\Omega} (a_m-a) (p-p_l) \ud x \bigg| & \leq ||a_m-a||_{0,\infty} \bigg|  \int_{\Omega} p-p_l \ud x \bigg| \\
&  \leq C (||a_m||_{0,\infty} + || a ||_{0,\infty}) ||p-p_l||_{0,2} \\
& \leq C  ||p-p_l||_{0,2} \to 0 \, \text{as} \, l \to \infty.
\end{align*}
This  concludes the proof.
\end{proof}


\section*{Acknowledgements}
The author is grateful to Wilfrid Gangbo, Michael Neilan, Ricardo H. Nochetto and Wujun Zhang for many useful discussions. 
This work began when the author was supported in part by a 2009-2013 Sloan Foundation Fellowship and continued while the author was in residence at
the Mathematical Sciences Research Institute (MSRI) in Berkeley, California, Fall 2013. The MSRI receives major funding from the National Science
Foundation under Grant No. 0932078 000. The author was partially supported by NSF DMS grant No 1319640.


\end{document}